\documentclass[reqno]{amsart}
\usepackage[utf8]{inputenc}
\usepackage{fancyhdr}
\usepackage{amsmath}
\usepackage{hyperref}
\usepackage{mathtools}
\usepackage{mathrsfs}
\usepackage{amsthm}
\usepackage{amsfonts}
\usepackage{stmaryrd}
\usepackage{tikz}
\usepackage{verbatim}
\usepackage[plain]{algorithm}
\usepackage{algpseudocode}
\usepackage{stmaryrd}
\usepackage{tikz-cd}
\usepackage{pgfplots}
\pgfplotsset{compat=1.12}
\usepgfplotslibrary{fillbetween}
\usetikzlibrary{patterns}
\usepackage{dsfont}
\usepackage[shortlabels]{enumitem}
\usepackage[colorinlistoftodos]{todonotes}
\usepackage{faktor}
\usepackage{enumitem}
\usepackage{amssymb}

\usepackage[
backend=biber,
style=numeric,
maxnames=5,
sorting=none
]{biblatex}

\newcommand{\R}{\mathbb{R}}
\newcommand{\Pro}{\mathbb{P}}

\newcommand{\Z}{\mathbb{Z}}

\newcommand{\N}{\mathbb{N}}

\newcommand{\diffvol}{\text{Diff}_{\text{vol}}^\infty(M)}
\newcommand{\diff}{{\text{Diff}^\infty(M)}}
\newcommand{\Conf}{\bold{CS}(TM)}
\newcommand{\conf}[1][y]{\bold{CS}(T_{#1}M)}
\newcommand{\Grass}{\bold{Gr}^k(TM)}
\newcommand{\grass}[1][y]{\bold{Gr}^k(T_{#1}M)}
\newcommand{\ext}[1][x]{\bigwedge^kT_{#1}M}
\newcommand{\Ext}{\bigwedge^kTM}

\theoremstyle{plain}
\newtheorem{thm}{Theorem}[section] 
\newtheorem*{thm*}{Theorem} 
\newtheorem*{thma*}{Theorem A} 
\newtheorem*{thmb*}{Theorem B} 

\newtheorem{lem}[thm]{Lemma} 
\newtheorem{cor}[thm]{Corollary} 
\newtheorem*{cor*}{Corollary} 

\theoremstyle{definition}
\newtheorem{defn}[thm]{Definition}

\theoremstyle{remark}

\title{Uniformly Expanding Random Walks on Manifolds}
\author{Rosemary Elliott Smith}

\begin{document}

\maketitle

\section{Introduction}
A random walk on a smooth manifold $M$ is defined by a probability measure $\mu$ on $\diff$-- each step of the walk is determined by a random $f\in \text{supp}(\mu)$, and independent of the previous steps. In this paper we will construct random walks on smooth manifolds with a particular property which we will call \textit{uniform expansion}.
\begin{defn}
    We say a probability measure $\mu$ on $\diff$ is \textit{uniformly expanding} if there exists a $C>0, N\in \N$ such that for all $x\in M$, and $v\in T^1_xM$, we have\footnote{Here we have $\mu^{(N)}:=\mu\ast\mu\ast\cdots\ast\mu$ ($N$ times).}
	$$\int_{\diff} \log ||D_xf(v)|| d\mu^{(N)}(f)>C$$
\end{defn}
This property tells us the dynamical system is, on average, expanding everywhere. For a subset of uniformly expanding random walks, the work of Brown, Eskin, Filip, and Rodriguez Hertz (\textit{in prep}, \cite{eskinpaper}) provides a strong Ratner-like classification theorem that completely determines the type of orbits that can exist; they are either dense or finite. 

\begin{thmb*}[\cite{eskinpaper}]
	Let $\mu$ be uniformly expanding in all dimensions, $\text{supp}(\mu)\subseteq \diffvol$. If all invariant measures of $\mu$ have non-zero Lyapunov exponents, then the only invariant measures of $\mu$ are volume or finitely supported.
	\end{thmb*}
	This has the immediate corollary: 
\begin{cor*}
Under the assumptions of Theorem B, the orbit of any point $x\in M$ under $\langle \text{supp}(\mu)\rangle$ is either dense or finitely supported.   
\end{cor*}	

	Further, these random walks exhibit unusual properties in other ways. For instance, all stationary measures of the random walk are actually invariant, and so any orbit supports an invariant measure, and all invariant measures are classified by Theorem B.
	
	\begin{thma*}[\cite{eskinpaper}]
	Let $\mu$ be uniformly expanding in all dimensions, $\text{supp}(\mu)\subseteq \diffvol$. Then any $\mu$-stationary measure on $M$ is $\mu$-invariant.  
	\end{thma*}
	
 Where classification theorems such as these exist, so do questions of existence-- how common are random walks that meet the criteria above, and in what settings can they exist? Here, we demonstrate the abundance of uniformly expanding random walks in smooth dynamical settings of any dimension. This abundance is perhaps surprising; if $\mu=\delta_{f_0}$ for $f_0\in\diffvol$, the random walk with law given by $\mu$ is never uniformly expanding. Indeed, in systems with great flexibility for movement, it is not immediately obvious why there would exist a surplus of random walks with extreme rigidity, but they are nonetheless present. 
 
 Potrie \cite{Potrie2021} showed that given any open set $U$ of $\text{Diff}_{\text{vol}}^\infty(\mathbb{T}^2)$, there exists an uniformly expanding random walk $\mu$ supported on a finite subset of $U$. In this paper we extend those results to closed manifolds of any dimension building on the work of Potrie \cite{Potrie2021} and Chung \cite{Chung2020}. Additionally, we use a stronger definition of uniformly expanding. This extension requires characterizing the growth of subspaces of $TM$ rather than single vectors-- a natural generalization in the higher dimensional cases. 
 
 	\begin{defn}\label{def: uniformly expan}
	Fix $k\in \{1,\cdots,d-1\}$. We say $\mu$ is \textit{uniformly expanding in dimension $k$} if there exists a $C>0, N\in \N$ such that for all $x\in M$, and $v\in \Pro(\ext[x])$ decomposable, we have
	$$\int_{\diff} \log ||D_xf(v)|| d\mu^{(N)}(f)>C$$
	\end{defn}
	This is an analogous definition to the one in \cite{Chung2020}. If $\mu$ is uniformly expanding in dimension $k$ for all $k\in \{1,\cdots, d-1\}$, we will call $\mu$ \textit{uniformly expanding in all dimensions}. Note that uniform expansion is a finite condition (see Lemma \ref{lem: finite support}).

	\begin{thm}\label{thm: big boi}
	Let $M$ be a closed smooth Riemannian manifold of dimension $d$, equipped with a volume form $\omega$. For any open $U\subset \diffvol$, there is a finitely supported measure $\mu$ on $\diffvol$ such that $\text{supp} (\mu)\subset U$, and $\mu$ is uniformly expanding in all dimensions in the sense of Definition $\ref{def: uniformly expan}$.
	\end{thm}
	
	Using Theorem B of \cite{eskinpaper}, the following corollary is immediate: 
	\begin{cor}
	For $\mu$ as in Theorem \ref{thm: big boi}, the $\mu$-stationary measures are invariant and are either volume or finitely supported.
	\end{cor}
	
	Another consequence of uniform expansion is positivity of top Lyapunov exponent. Using Theorem 2.2 of \cite{Chung2020} (extended to manifolds of any dimension), we obtain: 
	\begin{cor}
	For any open $U\subset \diffvol$, there is a finitely supported measure $\mu$ on $\diffvol$ such that $\text{supp} (\mu)\subset U$, and the top Lyapunov exponent of $\mu$ with respect to any stationary measure is positive.
	\end{cor}
 In an upcoming work, we hope to show all exponents are non-zero.

	\section{Setup and Rationale}
The structure of the paper is as follows: in Section 2, we recall some classical results, and formulate a criterion for detecting measures that are not uniformly expanding. In Section 3, we will use an invariance principle to sharpen this criterion and describe it in terms of non-random algebraic structures. In Sections 4 and 5, we will lay out a concrete construction of a uniformly expanding measure. Finally, in Section 6 we will use the tools of Section 3 to conclude.

For the remainder of the paper, let $M$ be a closed smooth manifold of dimension $d$ with Riemannian volume form $\omega$. The measure induced by $\omega$ will be denoted $vol_M$. On $M$, we denote by $\diff$ the set of smooth diffeomorphisms of $M$. Inside $\diff$ sits $\diffvol$, the set of smooth volume preserving diffeomorphisms on $M$. We let $\mu$ be a probability measure on $\diffvol$ with bounded support, and consider the random walk on $M$ defined by the law $\mu$. For the time being, this is an unspecified measure-- a more concrete construction will be provided later, and we will also denote this measure $\mu$. 
	
	One of our primary objects of study will be the stationary measures of $\mu$:
	
	\begin{defn}
	Let $\nu$ be a Borel probability measure on $M$. We call $\nu$ \textit{$\mu$-stationary} if $\nu =\mu\ast\nu$ in the sense of convolution of measures. A stationary measure $\nu$ is ergodic if given any $\nu$ measurable set $A\subset M$ such that $\nu(f(A)\Delta A)=0$ for $\mu$-a.e. $f$, then $\nu(A)=0$. \footnote{Here, $\Delta$ is set difference.}. 
	\end{defn} 
	
	$\mu$-stationary measures can be thought of as invariant on $\mu$-average-- they are a weaker analogue of invariant measures. Such measures often inform an understanding of the random walk defined by $\mu$, characterizing various properties. In particular, on the $\overline{\text{supp}(\mu)}$-orbit of any point $x\in M$, we can construct a $\mu$-stationary measure, and so understanding stationary measures furthers our understanding of orbits closures as well.   
	
	It is worth noting that $\mu$-stationary measures can also be characterized as follows: let $\sigma:\diff^\N\to\diff^\N$ be the left shift map, i.e. $\sigma(\alpha)_i=f_{i+1}$ for $i \in \N$, $\alpha=(f_0, f_1, \cdots)\in \diff^\N$\footnote{The $i$th coordinate of $\alpha\in \diff^\N$ is denoted by $\alpha_i$}. Then the $\mu$-stationary measures on $M$ are precisely the measures $\nu$ such that $\mu^\N\times \nu$ is invariant under the skew product map $$T\colon\diff^\N\times M\to \diff^\N\times M \text{ where } T(\alpha, x)=(\sigma(\alpha), f_0(x))$$ See \cite{bqbook} for details. This is a non-random dynamical system that characterizes $\mu$, $\nu$, and the relationship between them. We will study it often in the remainder of the paper.
	
	Finally, we set out the relevant spaces for our dynamics. In addition to the tangent bundle $TM$, here we will also study the change of $k$-dimensional subspaces of the tangent bundle for $k\in \{1,\cdots, d-1\}$. Thus, we extend the action of our dynamics to the $k$th Grassmanian of $T_xM$, denoted $\grass[x]$. Under the Pl\"ucker Embedding, any element of $\grass[x]$ can be viewed as a decomposable element of $\Pro(\ext)$, the projectivized $k$th exterior power of $TM$. Write $v$ for both an element of $\Pro(\Ext)$ and for an element of $\Ext$ that represents it. Here, decomposable means that it can be written as a single wedge product-- such elements span $\ext$, and precisely make up the image of $\grass[x]$ under the Pl\"ucker Embedding. It is often be useful to move between $\Grass$ and its image in $\Pro(\Ext)$, and we will use them interchangeably. We denote the full Grassmannian bundle over $M$ by $\Grass$-- at each point $x\in M$, the fiber is given by $\grass[x]$. Similarly, the bundle of conformal structures of $TM$ over $M$ will be denoted by $\Conf$, and at each point $x\in M$ the fiber is denoted by $\conf[x]$, the conformal structures on $T_xM$. As $M$ has a given Riemannian metric denoted by $\langle \cdot,\cdot\rangle$, each fiber $\conf[x]$ can be identified with the orbits of $\mathcal{S}_d$ under multiplication by $\mathcal{C}_d$, where $\mathcal{S}_d$ is the set of $d\times d$ symmetric positive definite matrices and $\mathcal{C}_d$ is the set of $d\times d$ conformal matrices, that is $$\mathcal{C}_d=\{A\in \text{GL}_d(\R)\ |\ A^\top A=c^2 I_d,\ \text{for some } c>0 \}$$ 

	Note that the standard norm of a decomposable vector $v=v_1\wedge\cdots\wedge v_k$ in $\ext$ is given by the determinant of the Gram matrix, $[\langle v_i,v_j\rangle]_{ij}$. This is the $k$-volume of the parallelotope spanned by $v_1,\cdots, v_k$; studying its change is an intuitive measurement of expansion or contraction of our dynamical system. In this setting, given $f\in \diff$ and $x\in M$, we extend the action of $D_xf\in \text{Hom}(T_xM, T_{f(x)}M)$ to $D_xf\in\text{Hom}(\ext[x], \ext[f(x)])$ by $$D_xf(v_1\wedge\cdots \wedge v_k)=D_xfv_1\wedge\cdots\wedge D_xfv_k$$ 
	
	Let $k\in \{1,\cdots, d-1\}$, $\nu$ be an ergodic $\mu$-stationary measure on $M$. Define $d_k=\binom{d}{k}$ for all $k$. Given an $f\in \diff$, we define $F_f^k:\Ext\to \Ext$ by $F_f^k(x,v)=(f(x), D_xf(v))$. We will suppress the $k$ when it is clear what exterior power we are considering. Consider the cocycle $$A_k\colon\diff^\N\times \Ext\to \diff^N\times \Ext$$ $$A_k(\alpha, (x,v))=(\sigma(\alpha), F_{f_0}(x,v))$$
    By Oseledets' Theorem \cite{Oseledets1968}, for $\nu$-a.e. $x\in M$ and $\mu^\N$-a.e. $\alpha\in \diff^\N$, there is a flag of subspaces $\Pro(\ext[x])=V_1(x,k)\supset\cdots \supset V_{d_k}(x,k)$ and a set of corresponding $\lambda_1(\nu, k)\ge\cdots \ge \lambda_{d_k}(\nu, k)$ such that for any $v\in V_i(x,k)\backslash V_{i+1}(x,k)$, $$\lim_{n\to\infty} \frac{1}{n}\log ||D_x f^n_\alpha v|| = \lambda_i(\nu, k)$$
	Here $f^n_\alpha=f_{n-1}\circ\cdots\circ f_0$. We will call $\lambda_i(\nu,k)$ the $i$th Lyapunov exponent of $A_k$ with respect to $\nu$, or simply the Lyapunov exponent of $\nu$ when the cocycle is clear.

	We say two cocycles $B,A_k\colon\diff^\N\times \Ext\to \diff^N\times \Ext$ are Lyapunov cohomologous by a transfer cocycle $C:\diff^\N\times \Ext\to \diff^N\times \Ext$ if $$B_{(\alpha,x)}=C_{T(\alpha,x)}^{-1}(A_k)_{(\alpha, x)} C_{(\alpha, x)}$$ where $B_{(\alpha,x)}\colon\ext[x])\to \ext[f_0(x)]$ is the associated map between fibers of the cocycle $B$. Note that this conjugation does not lose the information given by Oseledets' Theorem, as we may choose $C$ so that the Lyapunov exponents of $B$ and $A_k$ agree (see \cite{ACO1999} Proposition 8.2). 

	 Noting that part of Proposition 3.17 of \cite{Chung2020} can easily be extended to manifolds of any dimension, we can formulate a criterion for verifying that $\mu$ is uniformly expanding in all dimensions in terms of $\mu$-stationary measures on $\Grass$. We state the extended theorem and adjusted proof here.

	\begin{thm}\label{thm: Brian's}
If the measure $\mu$ is not uniformly expanding in all dimensions, then for some $k\in \{1,\cdots, d-1\}$, there is an ergodic $\mu$-stationary measure $\eta$ on $\Grass$ that has non-positive top Lyapunov exponent on $\Pro(\Ext)$. 
	\end{thm}

\begin{proof}
Assume that $\mu$ is not uniformly expanding in all dimensions, and fix $\epsilon>0$. Then $\exists k\in \{1,\cdots, d-1\}$ such that for all $n\in \N$, there exists $(x_n, v_n)\in \Pro(\bigwedge^{k}T_{x_n}M))$ such that $v_n$ is decomposable and $$\int \log||D_{x_n}f(v_n)||d\mu^{(n)}(f)<\epsilon$$
We will construct a limit of finite measures supported on the orbits of the points $(x_n,v_n)$ to build a measure $\eta$ that has non-positive exponent. For all $n$, define $$\eta_n:=\frac{1}{n}\sum_{m=0}^{n-1}\mu^{(m)}\ast \delta_{(x_n, v_n)}=\frac{1}{n}\sum_{m=0}^{n-1}\int\delta_{F_f(x_n,v_n)}d\mu^{(m)}(f)$$
Now, we evaluate the following to estimate how far from $\mu$-stationary $\eta_n$ is. $$\mu\ast\eta_n-\eta_n=\frac{1}{n}\sum_{m=0}^{n-1}\left[\int \delta_{F_f(x_n,v_n)}d\mu^{(m+1)}(f)-\int \delta_{F_f(x_n,v_n)}d\mu^{(m)}(f)\right]=\frac{1}{n}\left[\int \delta_{F_f(x_n,v_n)}d\mu^{(n)}(f)-\delta_{(x_n,v_n)}\right]$$
And so we have, $$\lim_{n\to \infty}\mu\ast\eta_n-\eta_n=0$$

Take $\eta$ to be a weak-$\ast$ limit of $\eta_n$. By the above we see that $\eta$ is a $\mu$-stationary measure on $\Grass$, and also induces a $\mu$-stationary measure on $\Pro(\Ext)$ by the pushforward of $\eta$.

Define $\Phi\colon\diff\times \Grass\to \R$ by $\Phi(g, (x,v))= \log ||D_xg(v)||$. Note that for $\alpha=(f_0,f_1,\cdots)$ we can write $$ \log || D_x f_\alpha^n(v)||\le \sum_{m=0}^{n-1}\Phi(f_m, F_{f_\alpha^m}(x,v))||$$
We can separate the variables $f_m$ and $f_\alpha^m$ in integration, and conclude the following for any $(x,v)$:
\begin{align*}
    \int\log ||D_xf(v)||\ d\mu^{(n)}(f) &= \int \log ||D_xf_\alpha^n(v)||\ d\mu^\N(\alpha) \\
    & \le \sum_{m=0}^{n-1}\int \Phi(f_m, F_{f_\alpha^m}(x,v))\ d\mu^\N(\alpha)\\
    &\le \sum_{m=0}^{n-1}\int \Phi(g, F_f(x,v))\ d\mu (g)\ d\mu^{(m)}(f)
    \end{align*}
    
Therefore, for any $n\in \N$, we have 
\begin{align*}
\int \int \log ||D_xg(v)|| d\mu (g)\ d \eta_n (x,v) &= \int \int \Phi(g, (x,v)) d\eta_n(x,v)\ d \mu (g)\\
&\le \frac{1}{n}\sum_{m=0}^{n-1}\int \int \Phi(g,F_f(x_n,v_n)) d\mu^{(m)}(f)\ d\mu (g)\\
&\le \frac{1}{n}\int \log||D_{x_n} f(v_n)|| d\mu^{(n)}(f) \\
& \le \frac{\epsilon}{n}
\end{align*}

By weak-$\ast$ convergence and the continuity of $\log$, we see $$\int\int \log ||D_xg(v)||\ d\mu(g) d\eta (x,v)\le 0$$

Possibly passing to an ergodic component of $\eta$ that preserves the above, we can assume $\eta$ is an ergodic $\mu$-stationary measure on $\Grass$ such that $\int \int \log ||D_xg(v)||\ d\mu (g)\ d\eta (x,v)\le 0$. Define $T\colon\diff^\N\times \Grass$ by $T(\alpha, (x,v))= (\sigma(\alpha), f_0(x,v))$. As $\eta$ is ergodic and $\mu$-stationary, we know $\mu^\N\times \eta$ is an ergodic $T$-invariant measure on $\diff^\N\times \Grass$. Thus, by the Birkhoff Ergodic Theorem, for $\mu^\N\times \eta$-a.e. $(\alpha,(x,v))$, we have
$$\lim_{n\to \infty}\frac{1}{n}\sum_{m=0}^{n-1}\Phi(f_0, f^n_\alpha(x,v)) =\int \int \Phi(f_0, (x,v))\ d\mu^\N(\alpha)\ d\eta (x,v) = \int \int \log ||D_x f(v)||\ d\mu(f)\ d\eta(x,v)$$
By the construction of $\eta$, this equality is non-positive, and so we can conclude 
$$\lim_{n\to\infty}\frac{\log || D_x f_\alpha^n(v)||}{n}\le 0$$

This directly shows that the Lyapunov exponents of $\eta$ are non-positive, and thus the theorem is proved.

\end{proof}
	
	\begin{defn}\label{def: invar}
	Let $\nu$ be a $\mu$-stationary measure on $M$, $\pi\colon\Grass\to M$ be projection on to $M$. We say $E\subseteq \Grass$ is a \textit{$\mu$-invariant $\nu$-measurable algebraic structure in $\Grass$} if:
	
	\begin{enumerate}
	    \item $E_x$ is an algebraic subset\footnote{By algebraic we mean ``defined by polynomial equations."} of $\grass[x]$ for $\nu$-a.e $x\in M$, where $E_x=E\cap \pi^{-1}(x)$ is the fiber of $E$ over $x$, and
	    \item for $\mu$-a.e. $f\in \diff$, $\nu$-a.e. $x\in M$, $E_{f(x)}=D_xfE_x$.
	\end{enumerate}
	
	Similarly, we say $C$ is a \textit{$\mu$-invariant $\nu$-measurable conformal structure on $TM$} if $C\colon M\to \text{CS}(TM)$ is a $\nu$-measureable map and for $\nu$-a.e. $x\in M$, $\mu$-a.e. $f\in \diff$, we have $C_{f(x)}=D_xfC_x$, where $C_x=C(x)$.
	\end{defn}

	\begin{cor}\label{cor: proper bundle}
	If $\mu$ is not uniformly expanding in all dimensions, then there is an ergodic $\mu$-stationary measure $\nu$ on $M$ and a $\mu$-invariant $\nu$-measurable algebraic structure of $\Grass$-- denoted $E$-- with non-positive top Lyapunov exponent in $\Pro(\Ext)$, for some $k\in \{1,\cdots, d-1\}$. Further, if the top Lyapunov exponent of $\nu$ on $\Pro(\Ext)$ is positive, then $E$ is proper. 
	\end{cor}
	\begin{proof}
	From Theorem \ref{thm: Brian's}, we know there is a $\mu$-stationary measure $\eta$ on $\Grass$ that has non-positive top Lyapunov exponent. Let $\nu:=\pi_\ast \eta$. Then $\nu$ is an ergodic $\mu$-stationary measure on $M$ as $\pi$ is equivariant and $\eta$ is is ergodic and $\mu$-stationary. 
	Let $\eta_x$ be the disintegration of $\eta$ along fibers over $x\in M$, and define $V_x:=\text{supp}(\eta_x)$, $V:=\cup_{x\in M}V_x$, and $V':=\cup_{x\in M}\text{span}(V_x)$. Note that as $V$ is $f$-invariant for $\mu$-a.e. $f\in \diff$, we see that $V\subseteq \Grass$ defines a $\mu$-a.e. $f$-invariant closed subset of $\Grass$. $V'$ forms a similarly invariant subbundle of $\Ext$. As $V'$ is an invariant subbundle of the $A_k$, there is a well-defined restriction to $V'$ with non-positive Lyapunov exponents with respect to $\nu$.
	
	Let $$E:=V'\cap \Grass$$
	
	As the fibers of $V'$ are subspaces, they are algebraic. Thus $E$ is a $\mu$-invariant $\nu$-measurable algebraic structure in $\Grass$ with non-positive Lyapunov exponents. 
	
	Finally, if $\nu$ has positive top Lyapunov exponent, then $V$ is a proper subset of $\Grass$, and so $V'$ must be a proper subbundle of $\Ext$, as decomposable elements span $\Ext$. If $E=\Grass$, then $\Grass\subseteq V'$, and so $E$ must be a proper subset of $\Grass$.  
	\end{proof}

	\section{Reduction to Studying Invariance}
	In this section, we will discuss how to reduce the problem of establishing uniform expansion to the study of invariant structures on $\Grass$. We start with a lemma. 
	
	\begin{lem}\label{lem: pos exp}
	Let $\nu$ be an ergodic $\mu$-stationary measure, $\text{supp}(\mu)\subset\diffvol$. If $\lambda_1(\nu,k)\le 0$ for some $k\in\{1,\cdots, d-1\}$, then $\lambda_i(\nu,1)=0$ for any $i$, i.e., the Lyapunov exponents of $A_1$ on $TM$ with respect to $\nu$ are zero. 
	\end{lem}
	
	\begin{proof}
	It suffices to prove the stronger statement that $ \lambda_1(\nu,1)\le 0$ if and only if $\lambda_1(\nu,k)\le 0$ for all $k\in \{1,\cdots, d-1\}$. As $\mu$ is supported on volume preserving diffeomorphisms, the Lyapunov exponents of $A_1$ with respect to $\nu$ on $TM$ must sum to 0. On the exterior product, the top Lyapunov exponent of $A_k$ is given by the sum of the top $k$ Lyapunov exponents of $A_1$ on $TM$. If $\lambda_1(\nu,1)\le 0$, then $\lambda_i(\nu,1)=0$ for all $i\in \{1,\cdots, d\}$, and so $\lambda_1(\nu,k)=0$. To prove the other direction, we show the converse: if $\lambda_1(\nu,1)>0$, then as the Lyapunov exponents are listed in decreasing order $\sum_{i=1}^{d-1}\lambda_i(\nu,1)=-\lambda_{d}(\nu,1)>0$. This means $\lambda_1(\nu,k)>0$ for any $k\in\{1,\cdots, d-1\}$, and so all exterior powers will have positive top exponent.
	\end{proof}
	
	 We will use a classical invariance principle of Ledrappier \cite{Ledrappier1986}, extended by Avila-Viana \cite{AvilaViana2010}, in the following theorem.

\begin{thm}\label{thm: led}
If $\mu$ is not uniformly expanding in all dimensions, then there is an ergodic $\mu$-stationary measure $\nu$ on $M$ such that one of the following is true:
\begin{enumerate}
\item for some $k\in\{1,\cdots, d-1\}$, there is a proper $\mu$-a.e. invariant $\nu$-measurable algebraic structure of $\Grass$, or
\item there is a $\mu$-a.e. invariant $\nu$-measurable conformal structure on $TM$. 
\end{enumerate}
\end{thm}

Using this theorem, we will show that all $\mu$-stationary measures have positive top Lyapunov exponent and do not admit a non-random invariant measurable algebraic structure $\Grass$ or a conformal structure on $TM$. It is with this fact that we will prove our measure is uniformly expanding. If $\lambda_1(\nu,k)>0$, Corollary \ref{cor: proper bundle} tells us we are in case one. $\lambda_1(\nu,1)=0$ is dealt with in Lemma \ref{lem: zero lyap}.

	\begin{lem}\label{lem: zero lyap}
	Let $\nu$ be an ergodic $\mu$-stationary measure on $M$. If $\lambda_1(\nu, k)\le 0$ for some $k\in\{1,\cdots, d-1\}$, then there is either a proper algebraic subbundle of $TM$ or a conformal structure on $TM$ that is $\mu$-a.e. invariant and $\nu$-measurable. 
	\end{lem}
	
	\begin{proof}
	Fix $k\in \{1,\cdots, d-1\}$ such that $\lambda_1(\nu, k)\le 0$. By Lemma \ref{lem: pos exp}, we may assume $k=1$ and all Lyapunov exponents on $TM$ are zero. We are now in the situation to apply the invariance principle, as our extremal Lyapunov exponents agree. To do this, we will first need to extend our dynamical system $T$ from the one sided shift to the two sided shift. Let $\hat{\sigma}\colon\diff^\Z\to \diff^\Z$ be the two sided left shift, i.e., $\hat{\sigma}(\alpha)_i=f_{i+1}$ and $\hat{\sigma}^{-1}(\alpha)_i=f_{i-1}$ for any $\alpha=(\cdots, f_{-1}, f_0, f_1,\cdots)$. We define the map $T$ as before, and its extension $\hat{T}$ by the following:
	
	\begin{align*}
	    T:\diff^\N\times &M\to \diff^\N\times M  & \hat{T}:\diff^\Z\times &M\to \diff^\Z\times M\\
	    T(\alpha, x)&=(\sigma(\alpha), f_0(x)) & \hat
	{T}(\alpha, x)&=(\hat{\sigma}(\alpha), f_0(x))
	\end{align*}

	We will work with the natural extension $\rho$ of the measure $\mu^\N\times \nu$ on $\diff^\N\times M$ to the two sided shift space $\diff^\Z\times M$. Unlike the one sided case, it is not immediate that $\mu^\Z\times \nu$ is invariant under $\hat{T}$ if $\nu$ is $\mu$-stationary.\footnote{In fact, this is generally not true.} Our first step will establish that $\rho=\mu^\Z\times \nu$ when $\lambda_1(\nu)=0$. Below are the fiber bundles we consider, and their corresponding invariant measures. Here, $\pi$ is projection from the first component. 
	$$ \begin{tikzcd}
&\mu^\N\times\nu\arrow[swap]{d}{\pi_\ast} &\diff^\N\times M\arrow[swap]{d}{\pi} & &\diff^\Z\times M\arrow[swap]{d}{\pi} &\rho\arrow[swap]{d}{\pi_\ast} \\%
&\mu^\N &\diff^\N &  &\diff^\Z &\mu^\Z
\end{tikzcd}
$$
As seen in Example 3.13 in \cite{AvilaViana2010}, $\rho$ has a unique characterization in this setting-- $$\rho_\alpha=\lim_{n\to\infty}f_0^{-1}f_{-1}^{-1}\cdots f_{-n+1}^{-1}\nu$$ As before, $\rho_{\alpha}$ is the disintegration of $\rho$ over $\alpha\in \diff^\Z$. In particular, $\rho$ is invariant under $\hat{T}$, and as $\nu$ is $\mu$-stationary, $\rho$ is independent of the future of the random walk. That is, the disintegration of $\rho$ along fibers over $\alpha\in \diff^\Z$ does not depend on $\alpha_+$, where $\alpha_+=(f_0,f_1, \cdots)$ and $\alpha_-=(\cdots, f_{-1})$. In concrete terms, this says that if $\alpha_-'=\alpha_-$ for $\alpha'\in \diffvol^\Z$, then $$\rho_{\alpha}=\rho_{\alpha'}$$ The Lyapunov exponents of $\rho$ and $\mu^\N\times \nu$ agree (also demonstrated in 3.13 in \cite{AvilaViana2010}), and are therefore zero. Further, $\rho$ projects down to $\mu^\Z$ on $\diff^\Z$ as it is the natural extension of $\mu^\N\times \nu$. Thus, we are in a situation to apply Theorem B of Avila-Viana. We recall it here for clarity, stated in our notation and setting.

\begin{thmb*}[\cite{AvilaViana2010}]
Let $(\diff^\Z, \mathcal{B}', \mu^\Z)$ be a probability space. Let $M$ be a Riemannian manifold, $\diff^\Z\times M$ be a measurable fiber bundle with fibers modeled on $M$, $\hat{T}$ be a smooth cocycle over $\hat{\sigma}$, $\mathcal{B}\subset \mathcal{B}'$ be a generating $\sigma$-algebra such that both $\hat{\sigma}$ and $\alpha\to \hat{T}_\alpha$ are $\mathcal{B}$-measurable mod 0. Let $\rho$ be a $\hat{T}$-invariant probability measure that projects down to $\mu^\Z$. If $\lambda_i(\rho)=0$ for all $i$, then any disintegration $\alpha\to \rho_\alpha$ is $\mathcal{B}$-measurable mod 0. 
\end{thmb*}
	
	It now only remains to define a choice of $\mathcal{B}$. Let $\pi_i:\diff^\Z\to \diff$ be projection from the $i$th coordinate of the product. Define 
	$$\mathcal{B}:=\left\langle\bigcap_{j=1}^\ell\pi^{-1}_{i_j}(U_{i_j})\ |\ \ell\in \N, i_1,\cdots, i_\ell\in \N \text{ pairwise distinct}, U_{i_j}\subset \diff \text{ open}\right\rangle$$
	This is the $\sigma$-algebra on $\diff^\Z$ generated by the cylinder sets of the non-negative coordinates of $\diff^\Z$. Under forward and backward iterates of $\hat{\sigma}$, $\mathcal{B}$ generates the product $\sigma$-algebra on $\diff^\Z$. It is clear that both $\hat{\sigma}$ and $\alpha\to f_0$ are $\mathcal{B}$-measurable. Thus, by Theorem B, $\alpha\to \rho_\alpha$ is also measurable. This implies that $\rho$ is invariant under the past of the random walk, i.e., only depends on the cylinder sets of the non-negative coordinates. That is, if $\alpha_+=\alpha'_+$, then $\rho_{\alpha}=\rho_{\alpha'}$. As we previously established that the disintegration of $\rho$ is independent of the future, a Hopf argument will conclude that it must be independent of the choice of $\alpha$ for $\mu^\Z$ almost every $\alpha\in \diff^\Z$, and thus, $\rho=\mu^\Z\times \nu$. The Hopf argument is straightforward in this case-- $\mu^\Z$ is a product measure and so has a local product structure, and $\hat{T}$ admits $s$ and $u$ holonomies that are constant on the local stable and unstable laminations, respectively. See Proposition 4.8 of \cite{AvilaViana2010} for full details.
	
	With this characterization, we may now turn our attention to structures in $TM$. By the argument above, we may assume our base is equipped with $\mu^\Z\times\nu$, and this measure is invariant under $\hat{T}$. We consider the adjusted cocycle $$\hat{A}_1:\diff^\Z\times T^1M\to \diff^\Z\times T^1M \text{,  where  }$$
	$$\hat{A}_1(\alpha, (x, v))= (\hat{\sigma}(\alpha), (f_0(x), D_xf_0v))$$ 
		We now consider two cases, based on the structure of our measure and cocycle. We say a cocycle $B:\diff^\Z\times TM\to \diff^\Z\times TM$ is conformal if $B_{(\alpha,x)}\in \mathcal{C}_d$ for $\mu^\Z\times \nu$-a.e. $(\alpha, x)$. 
		
		Case 1: Assume $\hat{A}_1$ is cohomologous to a conformal cocycle $B$ by a transfer cocycle $C$. Then we can define a measure $\beta$ on $\diff^\Z\times T^1M$ by $\beta_{(\alpha, x)}:=C_{(\alpha, x)}^{-1}m$, where $m$ is Lebesgue measure on $\R^d$, and $\beta = \int \beta_{(\alpha, x)}\ \text{d} \mu^\Z\times \nu$. Note that $\beta$ is invariant under $\hat{A}_1$ by construction, and as we may insist $C$ is $\mu^\Z\times \nu$-measurable, $\beta$ projects to $\mu^\Z\times\nu$ on $\diff^\Z\times M$. An identical argument as above-- using the $\sigma$-algebra generated by $\mathcal{B}\times \mathcal{M}$, where $\mathcal{M}$ is the Borel $\sigma$-algebra on $M$-- shows that the disintegration of $\beta$ is independent of the choice of $\alpha\in \diff^\Z$ (but not necessarily independent of the choice of $x\in M$). Thus, for any $(\alpha, x)\in \diff^\Z\times M$ we have a conformal structure on $TM$ given by $C_{(\alpha, x)}^{-1}$ of the standard euclidean metric that is $\mu$-invariant and $\nu$-measurable.
		
		Case 2: Assume $\hat{A}_1$ is not cohomologous to a conformal cocycle. Let $\gamma$ be any stationary measure of $\mu$ on $T^1M$ that projects to $\nu$ on $M$. By the same argument as above, we see that $\mu^\Z\times\gamma$ is an invariant measure of $\hat{A}_1$ on $\diff^\Z\times T^1M$. The Lyapunov exponents of $\nu$ on $T^1M$ are all zero, and thus we are in a situation to apply Theorem 1 of \cite{Ledrappier1986}, the original linear version of the invariance principle. An identical argument as above shows that the disintegration of $\gamma$ is independent of the choice of $\alpha\in \diff^\Z$.

We may normalize $\hat{A}_1$ so that $$(\alpha, (x,v))\to\left(\hat{\sigma}(\alpha), \left(f_0(x),\frac{D_xf_0v}{|\det(D_xf_0)|^{1/d}}\right)\right)$$ $\hat{A}_1$  and the normalized version share the same invariant measures, and so without loss of generality we may replace $\hat{A}_1$ with its normalized copy. A version of Furstenberg's Lemma in \cite{ACO1999} shows that $\gamma_{(\alpha,x)}$ is supported on the union of two proper subspaces of $T_xM$ for $\mu^\Z \times \nu$-a.e. $(\alpha, x)\in \diff^\Z\times M$. As $\gamma$ is not dependent on $\alpha$, we have that $\text{supp}(\gamma_{f(x)})= D_xf\text{supp}(\gamma_x)$ for $\mu$ a.e. $f\in \diff$ $\nu$ a.e. $x\in M$. This support is proper, and so the structure formed by $\text{supp}(\gamma)$ will be a proper $\mu$-invariant $\nu$-measurable algebraic structure. 
	\end{proof}
	
	\begin{proof}[Proof of Theorem \ref{thm: led}]
	By Corollary \ref{cor: proper bundle} and Lemma \ref{lem: pos exp}, there is an ergodic $\mu$-stationary measure $\nu$ on $M$ that has either $\lambda_1(\nu,k)\le 0$ for all $k\in \{1,\cdots,d-1\}$, or admits a proper $\mu$-invariant $\nu$-measurable algebraic structure of $\Grass$ for some $k\in \{1,\cdots, d-1\}$. In the first case, Lemma \ref{lem: zero lyap} shows that there is either a proper $\mu$-a.e. invariant $\nu$ measurable algebraic structure or a conformal structure on $TM$. Thus the theorem is proved. 
	\end{proof}
	
	With this reduction, we have pivoted from general criteria for uniformly expanding measures, to a specific characterization in terms of invariant structures in the dynamics. It now remains to construct our measure and show that it can have no such invariant objects. Our construction will hinge on picking a specific family of diffeomorphisms of $M$ that are strictly compatible with volume, in the sense that they can only coexist with invariant structures if the $\mu$-stationary measure in question is mutually singular with volume.

	\section{Studying Invariance}

	In the following section we will construct a set of specific maps $g_x^a\in \diff$, for any $x\in M$, that depend smoothly on $a$. We then demonstrate several key properties of $g_x^a$, and use the compactness of $M$ to restrict to a finite set of base points $x_1,\cdots, x_j\in M$. The goal of this section is to create $g_{x_1}^a, \cdots, g_{x_j}^a$ so that we may eliminate the invariant structures identified in the previous section.

	Let $g_x^a\in \diff$ be defined as follows. Fix an $R>1$, and let $B_h(0)$ denote the ball of radius $h$ in $\R^d$ for $h>0$. By Moser's Theorem (see 5.1.27 of \cite{katokhassel}), for any $x\in M$ there is a neighborhood $U_x'$ of $x$ and a $\phi_x\colon U_x'\to B_R(0)$ such that $\phi_x$ is a diffeomorphism, $\phi_x(x)=0$, and $\phi_x$ takes divergence free vector fields to divergence free vector fields. We then select a more restricted open neighborhood of $x$, denoted $U_x$, such that $U_x\subset U_x'$ and $\phi_x(U_x)=B_{1/2}(0)$. For fixed $a$, we will construct a volume preserving map $\psi_1^a$ on $B_R(0)$ that is affine on $B_{1/2}(0)$ and smoothly interpolates to the identity on $\partial B_R(0)$. Finally, we will construct $g_x^a$ by conjugating $\psi_1^a$ on $B_R(0)$ by $\phi_x$-- $g_x^a$ will be affine on $U_x$, and the identity outside of $U_x'$.
	
		\vspace{10pt}
		
	To begin, let $d':=d^d+d-1$. Set $a=(a_1,\cdots,a_d, a_{d+2}, \cdots a_{d^2+d})\in\R^{d}\times \R^{d^2-1}=\R^{d'}$, $b_a:=(a_1,\cdots, a_d)$ and consider 
	
	$$A'_a:=\begin{bmatrix}
	c_a & a_{d+2} & \cdots & a_{2d}\\
	a_{2d+1} & a_{2d+2} & \cdots & a_{3d}\\
	\vdots & & & \vdots\\
	a_{d^2+1} &\cdots & \cdots & a_{d^2+d}
	\end{bmatrix}$$
    
    where $c_a=-\sum_{i=2}^da_{id+i}$. The Lie group $G$ of affine transformations of $\R^d$ is given by $(d+1)\times (d+1)$ matrices of the form $\begin{pmatrix}A & b\\ 0 & 1\end{pmatrix}$, for $A\in \text{SL}_d(\R)$, and its action on $\R^d$ is given by $$\begin{pmatrix}A & b\\ 0 & 1\end{pmatrix}\cdot y =Ay+b.$$

    We will work in the open neighborhood of the identity in $\mathfrak{g}$, the Lie algebra of $G$, where the exponential map has inverse given by log. Note that for small $a$, we may choose $b'_a$ and $A_a$ so that 
  $$\exp \begin{pmatrix} A_a' & b_a' \\ 0 & 0\end{pmatrix}=\begin{pmatrix} A_a & b_a \\ 0 & 1\end{pmatrix}$$
  By the choice of $c_a$ we know $\text{tr}(A'_a)=0$, and so $\det (A_a)=1$. Moving forward, we will denote $\begin{pmatrix} A_a' & b_a' \\ 0 & 0\end{pmatrix}\in \mathfrak{g}$ by $X_a$.

 Take $X_a\in \mathfrak{g}$ for $a\in (-1,1)^{d'}$ small enough, and consider the flow defined on $\R^d$ given by $\phi_t^a(y)=e^{tX_a}\cdot y$. Note that at time $t=1$, $\phi_1^a(y)=A_ay+b_a$, and as $X_a$ is a trace free vector field, we know $\phi_t^a$ is a volume preserving diffeomorphism of $\R^d$. 
Without loss of generality, we choose $R$ to be large so that $A_a(B_1(0))+b_a\subset B_R(0)$ for any $a\in (-1,1)^{d'}$, that is, the affine transformation $\exp(X_a)$ maps $B_1(0)$ inside $B_R(0)$. 

Let $\alpha_a=\iota_{X_a}\omega$. Since $\alpha_a$ is a closed $(d-1)$-form there is a $(d-2)$-form $\chi_a$ such that $d\chi_a=\alpha_a$. Let $f\colon B_{R}(0)\to [0,1]$ be a smooth bump function such that $f(y)=1$ for $y\in B_1(0)$ and $f(y)=0$ on a neighborhood of $\partial B_{R}(0)$. Then $d(f\chi_a)=\iota_{Y_a}\omega$, where $Y_a$ is a trace free vector field that agrees with $X_a$ on $B_1(0)$ and vanishes on a neighborhood of $\partial B_{R}(0)$. Let $\psi_1^a(y)$ be the time one map of the flow of $d(f\chi)$ for small $a$. This map will agree with $\phi_1^a$ on $B_1(0)$, and the identity on $\partial B_{R}(0)$. Further, it is volume preserving for any given $a$ by the above. 
    
    Define $g_x^a(y)=\phi_x^{-1}\circ\psi_1^a\circ \phi_x(y)$ for any $y\in U_x$. By construction, the maps $g_x^a$ are volume preserving diffeomorphisms on $M$ with several key properties that we establish in the following theorem. 
    
    \begin{thm}\label{thm: transitivity}
    Let $x\in M$, $k\in \{1,\cdots, d-1\}$, $a\in (-1,1)^{d'}$. For fixed $y\in U_x$, $v\in \grass$, the map $$a\to (g_x^a(y), D_yg^a_x(v))\in \Grass$$
    is smooth and has surjective derivative at $a=0$.
    \end{thm}
   
    \begin{proof}
   As the question is local, we may conjugate by $\phi_x$ and assume $M=\R^d$, $U_x=B_{1/2}(0)$, and $g_x^a=\psi_1^a$. 
    
    Fix $y\in B_{1/2}(0)$, $y=(y_1,\cdots, y_d)$. On $B_{1/2}(0)$, $\psi_1^a(y)=\exp(X_a)\cdot y$. As $a\to X_a$ is invertible, we restrict to a neighborhood $\mathfrak{g}(\epsilon)$ of the identity of $\mathfrak{g}$. Thus we reduce to the study of the derivative of $F\colon \mathfrak{g}(\epsilon)\to \R^d\times\mathbf{Gr}^k(\R^d)$ where
    $$F(X_a)= (\exp(X_a)\cdot y, D_y\exp(X_a)\cdot y (v))=(A_ay+b_a, D_y(A_ay)(v))$$
    
    Note that the derivative of the exponential map at zero is the identity, and so the derivative of $F$ at zero is $D_{0}F\colon \mathfrak{g}\to \R^d\times T_{v}\mathbf{Gr}^k(\R^d)$,\footnote{Under the identification $T_0\mathfrak{g}(\epsilon)\simeq \mathfrak{g}$.} given by 
    $$D_0F(X_a)=(X_a\cdot (y_1,\cdots, y_d,1), X_a \cdot(1,\cdots, 1,0) (v))=(A'_ay+b_a', A_a'v).$$
    
    It remains to argue why this is surjective, but this is clear-- all but one of the entries of $A_a'$ can be changed independently between $(-1,1)$, and so $A_a'$ acts transitively on subspaces of dimension $k$. Further, translation by $b_a'$ is transitive, and so the entire map is surjective. 
    \end{proof}

These facts give us a local transitive property on the manifold and every grassmanian, as they imply the maps $a\to g_x^a, a\to D_yg_x^av$ are submersions. Finally, note that as $M$ is compact, we may choose a subset $g_{x_1}^a,\cdots, g_{x_j}^a$ such that the corresponding sets $U_{x_1},\cdots, U_{x_j}$ cover $M$. We are now equipped to begin constructing our measure in earnest.

	\section{Constructing the Measure}
	Take $f_0\in U\subset \diffvol$, for $U$ an open set. Let $0<\epsilon<1/2$ be such that $g_{x_i}^a\circ f_0\in U$ for all $a\in (-\epsilon,\epsilon)^{d'}, i\in \{1,\cdots, j\}$. Choose $p_i>0$ so $\sum_{i=0}^jp_i=1$.  
	Define
	$$\mu = Cp_0\delta_{f_0}+ C\sum_{i=1}^jp_i\int_{[-\epsilon,\epsilon]^{d^2+d}}\delta_{g_x^a\circ f_0}da,$$

	where $C>0$ is a normalizing constant that makes $\mu$ a probability measure. Going forward, we will absorb the $C$ into $p_i$. Observe that $\mu$ can be split into an outer integral over $a_{d+2},\cdots, a_{d^2+d}$ and an inner integral over $a_1,\cdots, a_{d}$. This inner integral corresponds to the pure translation portion of $g_{x_i}^a$, and will be used to force additional regularity with respect to volume.

	\begin{thm}\label{thm: abs cont}
	Let $\mu$ be as defined above, and $\nu$ be a $\mu$-stationary probability measure. Then $\nu= \rho+ \nu'$, where $\rho, \nu'$ are positive measures such that $\rho<< \text{vol}_M$.   
	\end{thm}
	
	\begin{proof}
	Let $\nu$ be a $\mu$-stationary probability measure on $M$. We will show that convolution by $\mu$ produces a non-zero absolutely continuous part of $\nu$.

	For fixed $z\in M$, we know that $f_0(z)\in U_{x_k}$ for at least one $k$. Fix such a $k$. For $y\in U_{x_k}$, we have $g_{x_k}^a(y)=\phi_{x_k}^{-1}\circ\phi_1\circ \phi_{x_k}(y)$. Let $a_{d+2},\cdots, a_{d'}$ be fixed, $A_a$ be as above, and let $b_a=(a_1,\cdots, a_d)$ be variable. By construction, $A_a\phi_{x_k}(f_0(z))\in B_{1/2}(0)$, and so $A_a\phi_{x_k}(f_0(z))+b_a\in B_1(0)$. Where $\phi_{x_i}(f_0(z))$ is well defined, let $$B_i(z):=\phi_{x_i}^{-1}(A_a\phi_{x_i}(f_0(z))+[-\epsilon,\epsilon]^{d}).$$
	We will bound the inner integral from below. As $\omega$ is a Riemannian volume form, there is a smooth positive $g:U_{x_k}'\to \R$ such that $\omega = g dx_1\wedge\cdots dx_d$ on $U_{x_k}'$. Fix $c_k>0$ such that $c_k^{-1}<g<c_k$ on $U_{x_k}$. Note that for any such $A_a$, these constants and the compactness of $M$ imply that $vol_M(B_i(z))$ will be uniformly bounded from below and above by a constant not depending on $z,A_a$. Fix $A_a$. Then for $E\subseteq M$ measurable, we have, 
	
	\begin{align*}
	    \sum_{i=1}^jp_i\left(\int_{[-\epsilon,\epsilon]^{d}}\delta_{g_{x_i}^a\circ f_0(z)}(E)da_1\cdots da_d\right) &= \sum_{i=1}^jp_i\left(\int_{[-\epsilon,\epsilon]^{d}}\mathbf{1}_E(\phi^{-1}_{x_i}\circ\psi_1^a\circ\phi_{x_i}\circ f_0(z))da_1\cdots da_d\right)\\
	     &\ge p_k\int_{[-\epsilon,\epsilon]^{d}}\mathbf{1}_E(\phi^{-1}_{x_k} (A_a\phi_{x_k}(f_0(z))+b_a))db_a\\
	     &\ge p_k\int_{B_k(z)}\frac{\mathbf{1}_E(y)}{g(y)}d\text{vol}_M(y)\\
	     &\ge \frac{p_k}{c_k}\text{vol}_M(E\cap B_k(z)).
	\end{align*}
	
	We now consider $\mu\ast\nu$. Note that $$\mu=p_0\delta_{f_0}+\int_{[-\epsilon,\epsilon]^{d^2-1}}\sum_{i=1}^jp_i\left(\int_{[-\epsilon,\epsilon]^{d}}\delta_{g_{x_i}^a\circ f_0}da_1\cdots da_d\right)da_{d+2}\cdots da_{d^2+d}.$$ For any $\nu$ we may write $\nu(E)=\int_M\delta_z(E)d\nu(z)$. By Fubini's Theorem and the fact $\mu\ast\nu=\nu$, we have, $$\int_M \delta_z(E) d\nu(z)=\nu(E)=\mu\ast \nu (E)=\int_U\int_M\delta_{h(z)}(E)d\nu(z)d\mu(h)=\int_M\int_U\delta_{h(z)}(E)d\mu(h)d\nu(z).$$
Since this holds for any $E$ measurable, it follows that for any $E$ we have $\nu_z(E)=\int_U \delta_{h(z)}(E)d\mu(h)$, where $\nu=\int\nu_z\ d\nu(z)$. It then suffices to show $\nu_z$ has an absolutely continuous part for any given $z$. But \begin{align*}
    \nu_z(E) &= \int_U\delta_{h(z)}(E)d\mu(h)\\
    &= p_0\delta_{f_0(z)}(E)+ \sum_{i=1}^jp_i\int_{[-\epsilon,\epsilon]^{d'}}\delta_{g_x^a\circ f_0(z)}(E)da \\
    &\ge \int_{[-\epsilon,\epsilon]^{d^2-1}}\sum_{i=1}^j p_i\int_{B_i(z)}\frac{\mathbf{1}_E(y)}{g(y)}\cdot \mathbf{1}_{U_i}(f_0(z))d\text{vol}_M(y) dA_a\\
    &\ge \int_{[-\epsilon,\epsilon]^{d^2-1}}\max_i\left\{\frac{p_i}{c_i} \text{vol}_M(E\cap B_i(z))\right\}dA_a.
\end{align*}

 Let $\rho_z(E):=\int_{[-\epsilon,\epsilon]^{d^2-1}}\max_i\left\{\frac{p_i}{c_i} \text{vol}_M(E\cap B_i(z))\right\}dA_a$. Then $\rho_z$ is absolutely continuous with respect to volume. If $\text{vol}_M(E)>0$, then there is a positive volume set of $z$ such that $\rho_z(E)$ is bounded from below by a uniform constant not depending on $z$. Thus, we can conclude.
	\end{proof}
	
	\section{Proving Uniform Expansion}

	\begin{thm}\label{thm: no abscont1}
	Let $\nu$ be a $\mu$-stationary probability measure on $M$ such that $\nu=\nu_0+\nu_\perp$, where $\nu_0, \nu_\perp$ are positive measures and   $\nu_0\ll \text{vol}_M$. Then there are no $\nu$-measurable proper algebraic $\mu$-invariant structures of $\Grass$.
	\end{thm}
	
	\begin{proof}
We will proceed by contradiction, assuming there is a $\nu$-measurable $\mu$-invariant algebraic structure $E$ and showing $\nu_0=0$. Given an algebraic subset $E_x$ of $\Grass$, there is a corresponding subspace $V_x=\langle E_x\rangle$ generated by its image in $\Ext$. As $D_xfE_x=E_{f(x)}$ for $\nu$-a.e. $x\in M$, $\mu$-a.e. $f\in \diff$, Definition \ref{def: invar} implies $h=h(x):=\dim V_x$ is constant $\nu$-almost everywhere. Thus, any $\mu$-invariant $\nu$-measurable algebraic structure $E$ in $\Grass$ can be written as a $\nu$-measurable $\mu$-invariant map $E\colon M\to \textbf{Gr}^h(\Ext)$, for some fixed $h\in \N$. By the construction, $\mu$-invariance implies $E$ is invariant under the action of both $g_{x_i}^a\circ f_0$ and $f_0$ for any $i$, $m$-a.e. $a\in (-\epsilon,\epsilon)^{d'}$. This implies $E$ is invariant under $g_{x_i}^a$ for any $i$, $m$-a.e. $a\in (-\epsilon,\epsilon)^{d'}$.

	Let $\epsilon>0$. By Lusin's Theorem, there is a compact set $K_\epsilon\subset M$ such that $\nu_0(K_\epsilon)\ge (1-\epsilon)\nu_0(M)$ and $E$ is continuous when restricted to $K_\epsilon$. As $\nu_0$ is absolutely continuous with respect to volume, we may assume $K_\epsilon$ has no isolated points. As $M$ is compact, there is some $U_{x_k}$ such that $\nu_0(K_\epsilon\cap U_{x_k})>0$. For any $y\in K_\epsilon\cap U_{x_k}$, we may choose $a'\in (-\epsilon,\epsilon)^{d'}$ such that $g_{x_k}^{a'}(y)=y$ and $D_yg_{x_k}^{a'} E(y)\ne E(y)$. This is due to Theorem \ref{thm: transitivity} and the fact $E(y)$ corresponds to a proper closed subset of $\grass[y]$, so surjectivity on a neighborhood of $\grass$ lets us perturb proper subspaces. 

Let $$L:=\{a\in (-\epsilon,\epsilon)^{d'}\ |\ E(g_{x_k}^{a}(y))=D_yg_{x_k}^{a}E(y)\}$$ 
and define $L':=\{z\in U_{x_k}\ |\ z=g_{x_k}^a(y), a\in L\}$. We know $m(L)=1$, so by Theorem \ref{thm: transitivity}, we have $\text{vol}_M(L'\cap K_\epsilon)=\text{vol}_M(K_\epsilon\cap U_{x_k})$. As $\nu_0$ is absolutely continuous with respect to $vol_M$,  $\nu_0(L')>0$.
 Thus, there is a sequence $a_n\to a'$ such that $a_n\in L$ and $g_{x_k}^{a_n}(y)\in K_\epsilon$. As $E$ is continuous on $K_\epsilon$ and $a\to g_{x_k}^{a}$ is continuous, we have $$D_yg_{x_k}^{a_n}E(y)=E(g_{x_k}^{a_n}(y))\to E(g_{x_k}^{a'}(y))=E(y)$$ but $$D_yg_{x_k}^{a_n}E(y)\to D_yg^{a'}_{x_k}E (y)\ne E(y),$$ a contradiction. Thus, $\nu_0=0$. 
	\end{proof}
	
	\begin{thm}\label{thm: no abscont2}
	Let $\nu$ be a $\mu$-stationary probability measure on $M$ such that $\nu=\nu_0+\nu_\perp$, where $\nu_0, \nu_\perp$ are positive measures and   $\nu_0\ll \text{vol}_M$. Then there are no $\nu$-measurable $\mu$-invariant conformal structures on $TM$.
	\end{thm}
	
	\begin{proof}
	For the sake of contradiction, assume there is a map $C\colon M\to \Conf$ that is $\mu$-invariant and $\nu$-measurable. We will show $\nu_0=0$. Let $\epsilon>0$. By Lusin's Theorem, there is a compact set $K_\epsilon\subset M$ such that $\nu_0(K_\epsilon)\ge (1-\epsilon)\nu_0(M)$ and $C$ is continuous when restricted to $K_\epsilon$. As $\nu_0$ is absolutely continuous with respect to volume, we may assume $K_\epsilon$ has no isolated points. As $M$ is compact, there is some $U_{x_k}$ such that $\nu_0(K_\epsilon\cap U_{x_k})>0$. 
 
 Note that the image of $a\to A_a$ is a neighborhood of the identity in $\text{SL}_d(\R)$, so for some choice of $a$, the corresponding matrix $A_a$ is not conformal. Thus $D_yg^a_{x_i}$ will not preserve a conformal structure on $T_yM$ for any $y\in U_{x_k}$. Arguing as in the proof of Theorem \ref{thm: no abscont1}, we obtain that for any $y\in K_\epsilon$ there is an $a'\in (-\epsilon,\epsilon)^{d'}$ such that $g_{x_k}^{a'}(y)=y$ and $D_yg_{x_k}^{a'} C(y)\ne C(y)$. The remainder of the proof is identical to the proof of \ref{thm: no abscont1}. 
	\end{proof}
	
	With these two theorems we now have all the machinery in place to show that $\mu$ is uniformly expanding. 
	
	\begin{thm}
	The measure $\mu$ is uniformly expanding in all dimensions. 
	\end{thm}
	
	\begin{proof}
	Assume for the sake of contradiction that $\mu$ is not uniformly expanding in some dimension $k$. By Theorem \ref{thm: led}, there is a $\mu$-stationary measure $\nu$ such that there is either a $\mu$-invariant $\nu$-measurable algebraic structure on $\Grass$ or a $\mu$-invariant $\nu$-measurable conformal structure on $TM$. By Theorem \ref{thm: abs cont}, $\nu$ must have an absolutely continuous part with respect to volume. But by Theorems \ref{thm: no abscont1} and \ref{thm: no abscont2}, absolute continuity prevents the existence of invariant structures. Thus, the theorem is proved. 
	
	\end{proof}
	Finally, let us show that uniform expansion in all dimensions is an open property-- thus, we can discretize $\mu$ to construct a measure with finite support that is also uniformly expanding. 
	\begin{lem}\label{lem: finite support}
	Let $\mu_0$ be a probability measure on $\diff$ that is uniformly expanding in dimension $k\in \{1,\cdots, d-1\}$. Assume that $\text{supp}(\mu_0)$ is contained in some compact set $K$. Then there is an open neighborhood $V$ of $\mu_0$ in the weak-$\ast$ topology of $P(K)$ so that any $\rho\in V$ is also uniformly expanding in dimension $k$. Here, $P(K)$ is the set of probability measures on $K$.
	\end{lem}
\begin{proof}
We know $\exists C>0, N>0$ such that for any $(x,v)\in \Grass$, we have $$\int_K\log ||D_xfv||\text{d}\mu_0^{(N)}(f)>C$$ As $K$ is compact, $\phi_{(x,v)}(f):=\log ||D_xfv||$ is continuous and so uniformly bounded for $f\in K$. Define $$\phi\colon K\times \Grass\to \R \text{ by } \phi(f,(v,v)):=\phi_{(x,v)}(f)$$ This is a continuous function in all variables, and the domain is compact, so the family of continuous functions from $\Grass\to \R$ given by $\{\phi\circ f\ |\ f\in K\}$ is uniformly equicontinuous. By compactness of $K$ and $M$, there exists a finite cover $B_1,\cdots, B_\ell$ of $\Grass$ by open balls of fixed radius $\delta>0$ centered at $(x_1,v_1),\cdots (x_\ell,v_\ell)$ in $\Grass$ such that for all $i$, for any $(x,v)\in B_i$: 
\begin{equation*}
    |\phi_{(x_i,v_i)}(f)-\phi_{(x,v)}(f)|<\frac{C}{4} \text{ for all }f\in K,  \tag{$\star$} \label{equicont}
\end{equation*}
 and 
\begin{equation*}
    \left| \int \phi_{(x_i,v_i)}(f)d\rho^{(N)}-\int \phi_{(x,v)}(f)d\rho^{(N)} \right|<\frac{C}{4}  \text{ for all }\rho\in P(\diff). \tag{$\star\star$} \label{equicont2}
\end{equation*} 

We now define the sets $$V_i:=\left\{\rho^{(N)}\in P(K)\ |\ \left|\int_K\phi_{(x_i,v_i)}(f)\rho^{(N)}(f)- C\right| <\frac{C}{4}\right\}$$
These are open non-empty sets in the weak-$\ast$ topology by definition, and comprise of all measures that meet the criteria for uniform expansion on the point $(x_i,v_i)$. By \eqref{equicont} and \eqref{equicont2}, for any $(x,v)\in B_i$, $\rho^{(N)}$ is uniformly expanding at the point $(x,v)$ by constant $\frac{C}{2}$. Thus, $\rho\in V:=\cap_{i=1}^\ell V_i$ is uniformly expanding at all points in $M$. As $V$ is the finite intersection of non-empty open sets and $\mu_0\in V_i$ for all $i$, it too is a non-empty open set, and the claim is shown.

\end{proof}
Finally, we may prove the main theorem of the paper:

\begin{proof}[Proof of Theorem \ref{thm: big boi}]
Let $\mu$ be the measure constructed in Section 5. Consider a sequence $\mu_n$ of finitely supported measures such that $\text{supp}(\mu_n)\subset \text{supp}(\mu)$ for any $n$, and the weak star limit of $\mu_n$ is $\mu$. Note that by construction, the support of $\mu$ is contained (up to scaling) in the compact set $$\{g^a_{x_i}\circ f_0\ |\ a\in [-\epsilon,\epsilon]^{d'}, 1\le i\le d-1\}$$ By applying Theorem \ref{lem: finite support} for each dimension $k$ in $1,\cdots, d-1$, we see there is an open neighborhood $V$ of $\mu$ such that every measure in $V$ is uniformly expanding in all dimensions. As $\mu_n\to \mu$, the sequence must eventually enter the set $V$, and thus we may conclude. 
\end{proof}

\printbibliography

\end{document}